\documentclass[revision]{FPSAC2020}

\usepackage[english]{babel}
\usepackage{url}
\usepackage[utf8x]{inputenc}
\usepackage{tikz}
\usepackage{shuffle}
\usepackage{cleveref}
\usepackage{multicol}

\tikzset{style={anchor=base,baseline={([yshift=-1ex]current bounding box.center)}}}

\definecolor{mygreen}{RGB}{23,103,1}
\definecolor{Zgris}{rgb}{0.85,0.85,0.85}

\newcommand{\red}[1]{\textcolor{red}{#1}}
\newcommand{\blue}[1]{\textcolor{blue}{#1}}

\newcommand{\eqdef}{\mbox{\,\raisebox{0.2ex}{\scriptsize\ensuremath{\mathrm:}}\ensuremath{=}\,}}

\newcommand{\NN}{\mathbb{N}}

\newcommand{\RR}{\mathbb{R}}

\newcommand{\PP}{\mathbb{P}}

\newcommand{\PW}{\mathbf{PW}}

\newcommand{\SA}{\mathcal{A}}
\newcommand{\SP}{\mathcal{P}}
\newcommand{\ST}{\mathcal{T}}

\newcommand{\PForest}{\mathfrak{F}}
\newcommand{\PTree}{\mathfrak{T}}
\newcommand{\PParticular}{\mathfrak{P}}

\newcommand{\FQSym}{\mathbf{FQSym}}
\newcommand{\WQSym}{\mathbf{WQSym}}
\newcommand{\PQSym}{\mathbf{PQSym}}

\newcommand{\Ker}{\operatorname{Ker}}
\newcommand{\Node}{\operatorname{Node}}
\newcommand{\Prim}{\operatorname{Prim}}
\newcommand{\TPrim}{\operatorname{TPrim}}

\newcommand{\gcdot}{\slash}
\newcommand{\dcdot}{\backslash}
\newcommand{\ins}{\blacktriangleright}

\crefname{coro}{Corollary}{Corollaries}
\crefname{defi}{Definition}{Definitions}
\crefname{lem}{Lemma}{Lemmas}

\newtheorem{theorem}{Theorem}
\newtheorem{lem}[theorem]{Lemma}

\newtheorem{prop}[theorem]{Proposition}

\theoremstyle{definition}
\newtheorem{defi}[theorem]{Definition}

\theoremstyle{remark}
\newtheorem{ex}[theorem]{Example}
\newtheorem{rem}[theorem]{Remark}

\tikzstyle{alert} = [color=red, line width = 1.5]
\tikzstyle{bluealert} = [color=blue, line width =1.5]
\tikzstyle{big} = [line width = 1.5]
\tikzstyle{Point} = [fill, radius=0.08]
\tikzstyle{RedPoint} = [fill, radius=0.09, color = red]
\tikzstyle{Leaf} = [color = gray]

\tikzstyle{Red} = [color = red]
\tikzstyle{Blue} = [color = blue]
\tikzstyle{Green} = [color = mygreen]
\tikzstyle{Gray} = [color = gray]

\newcommand\Item[1][]{
  \ifx\relax#1\relax  \item \else \item[#1] \fi
  \abovedisplayskip=0pt\abovedisplayshortskip=0pt~\vspace*{-\baselineskip}}

\title{Basis of totally primitive elements of $\WQSym$}

\author{Hugo MLODECKI \thanks{\href{mailto:hugo.mlodecki@lri.fr}{hugo.mlodecki@lri.fr}}}

\address{Université Paris-Saclay, CNRS, Laboratoire de recherche en informatique, 91405, Orsay, France.}

\received{\today}

\abstract{By Foissy's work, the bidendriform structure of the Word
  Quasisymmetric Functions Hopf algebra (WQSym) implies that it is isomorphic to
  its dual.  However, the only known explicit isomorphism does not respect the
  bidendriform structure. This structure is entirely determined by so-called
  totally primitive elements (elements such that the two half-coproducts are 0).
  In this paper, we construct a basis indexed by a new combinatorial family
  called biplane forests in bijection with packed words.  In this basis,
  primitive elements are indexed by biplane trees and totally primitive elements
  by a certain subset of trees.  Thus we obtain the first explicit basis for the
  totally primitive elements of WQSym.}

\resume{Grâce aux travaux de Foissy, on sait que l'algèbre de Hopf WQSym est
  isomorphe à sa duale car bidendriforme. Cependant, le seul isomorphisme
  explicite connu ne respecte pas la structure bidendriforme. Cette structure
  est entièrement déterminée par les éléments totalement primitifs (annulés par
  les demi co-produits).  Dans ce papier, nous construisons une base indexée par
  une nouvelle famille combinatoire appelée forêt biplanes, en bijection avec
  les mots tassés. Dans cette base, les éléments primitifs sont indexés par les
  arbres et les totalement primitifs par un certain sous-ensemble
  d'arbres. Ainsi, nous obtenons la première base explicite des éléments
  totalement primitifs dans WQSym.}


\begin{document}

\maketitle

\section*{Introduction}

The varied zoo of combinatorial Hopf algebras is much better understood when one
considers the extra algebraic structures that each algebra can have. In this
light, operad related theories are very useful. Some important examples are the
Hopf algebras of non-commutative and quasi-symmetric functions related to the
theory of free Lie algebras~\cite{NCSF1}. More recently, the Hopf algebra of
binary trees was identified as the free dendriform algebra on one
generator~\cite{LodRon_PBT}.

Closed related examples include the algebras $\FQSym$ of permutations of
Malvenuto-Reutenauer~\cite{MalReu} and the Hopf algebra $\WQSym$ of surjections
or, equivalently, ordered set partitions~\cite{Hivert_thesis}. These two can be
seen as noncommutative versions of the algebra of quasi-symmetric
functions. Though the first one is trivially self-dual, it is only by a deep
theorem of Foissy~\cite{foissy_2007} that one can show that the second one is
too. In particular, until Vargas's work~\cite{Vargas_thesis}, no concrete
isomorphism was known.

The first study of $\WQSym$ structure is due to
Bergeron-Zabrocki~\cite{berzag}. They showed that it is free and
co-free. Independently, Novelli-Thibon endowed it with a bidendriform bialgebra
structure~\cite{novthi_2006}. Recall that a dendriform algebra is an abstraction
of a shuffle algebra where the product is split in two half-products. If the
coproduct is also split, and certain compatibilities hold, one gets the notion
of bidendriform bialgebra~\cite{foissy_2007}.

Building on the work of Chapoton and Ronco~\cite{ronco_2000,chapoton_2002},
Foissy~\cite{foissy_2007} showed that the structure of a bidendriform bialgebra
is very rigid. In particular, he defined a specific subspace called the space of
totally primitive elements, and showed that it characterizes the whole
structure. This does not only re-prove the freeness and co-freeness, as well as
the freeness of the primitive lie algebra, but also shows that the structure of
a bidendriform bialgebra depends only on its Hilbert series (the series of
dimensions of its homogeneous components). In particular, any such algebra is
isomorphic to its dual. However, Foissy's isomorphism is not fully explicit and
depends on a choice of a basis of the totally primitive elements. To this end,
one needs an explicit basis of the totally primitive elements. Foissy described
such a construction for $\FQSym$~\cite{foissy_2011}. Our long term goal is to
effectively apply Foissy's construction to $\WQSym$ in order to build an
explicit isomorphism respecting the bidendriform structure. With this objective
in mind, we provide an explicit basis called the totally primitive elements
using a bijection with certain families of trees.

We begin with a background section presenting two rigidity structure theorems
that prove, among other things, the self-duality of any bidendriform bialgebra
(\cref{brace,prim_tot}). We then define the notion of packed word as well as the
specific basis of~$\WQSym$, which will be the starting point of our
combinatorial analysis (\cref{half_prod,left_coprod,right_coprod}).

\cref{sect2_pw_trees} is devoted to the combinatorial construction of biplane
forests (\cref{forest_def}) which are our first key ingredient. They record a
recursive decomposition of packed words according to their global descents
(\cref{gd_fact}) and positions of the maximum letter (\cref{phi_fact}). We show
that the cardinalities of some specific sets of biplane trees match the
dimensions of primitive and totally primitive elements (\cref{thm:bij}).

Finally in \cref{sect3_Pbasis} we construct a new basis of $\WQSym$ which
contains a basis for the primitive and totally primitve elements
(see~\cref{thm:P}). To do so we decompose the space of totally primitive
elements as a certain direct sum which matchs the combinatorial decomposition of
packed words (\cref{tot_stable}).


\section{Background}

\subsection{Cartier-Milnor-Moore theorems for Bidendriform bialgebras}

A bialgebra is a vector space over a field $K$, endowed with an unitary associative
product $\cdot$ and a counitary coassociative coproduct
$\Delta$ satisfying a compatibility relation called the Hopf relation
$\Delta(a\cdot b) = \Delta(a)\cdot\Delta(b)$.  In this paper all bialgebras are assumed to be graded
and connected (\emph{i.e.} the homogeneous component of degree $0$ is $K$). They
are therefore Hopf algebras, as the existence of the antipode is implied.

We now recall the elements of the definition of bidendriform bialgebras which are
useful for the comprehension of this paper. We refer to~\cite{foissy_2007} for
the full list of axioms.

First of all, a \emph{dendriform algebra} $A$ is a $K$-vector space,
endowed with two binary bilinear operations $\prec$, $\succ$ satisfying the
following axioms, for all $a, b, c\in A$:
\begin{align}
  \label{E1} (a \prec b) \prec c &= a \prec (b \prec c + b \succ c),\\
  \label{E2} (a \succ b) \prec c &= a \succ (b \prec c),\\
  \label{E3} (a \prec b + a \succ b) \succ c &= a \succ (b \succ c).
\end{align}
Adding together~\cref{E1,E2,E3} show that the product
$a \cdot b \eqdef a \prec b + a \succ b$ is associative. Adding a subspace of scalars,
this defines a unitary algebra structure on $K \oplus A$.  In this paper,
all the dendriform algebras are graded and have null $0$-degree component so
that the associated algebra is connected.

Dualizing, one gets a notion of \emph{co-dendriform co-algebra} which is a
$K$-vector space with two binary co-operations (\emph{i.e.}, linear maps
$A \to A\otimes A$) denoted by $\Delta_\prec$, $\Delta_\succ$ satisfying the
dual axioms of~\cref{E1,E2,E3}.  The sum of the two half
coproducts~$\tilde{\Delta}(a) \eqdef \Delta_\prec(a) + \Delta_\succ(a)$ is a
reduced coassociative coproduct. On $K \oplus A$, setting 
$\Delta(a) \eqdef 1 \otimes a + a \otimes 1 + \tilde{\Delta}(a)$ defines a
co-associative and co-unitary coproduct.

A \emph{bidendriform bialgebra} is a $K$-vector space which is both a dendriform
algebra and a co-dendriform co-algebra satisfying a set of four
relations~\cite{foissy_2007} relating respectively $\prec$ and $\succ$ with
$\Delta_\prec$, $\Delta_\succ$. Adding those four relations shows that $\cdot$ and
$\Delta$ as defined above defines a proper bi-algebra.

\bigskip

We recall here the relevant results of Foissy~\cite{foissy_2007} on the
rigidness of bidendriform bialgebras based on the works of Chapoton and
Ronco~\cite{ronco_2000,chapoton_2002}.

Let $A$ be a bidendriform bialgebra. We define~$\Prim(A)\eqdef
\Ker(\tilde{\Delta})$ as the set of \emph{primitive} elements of $A$. We also denote
by~$\SA(z)$ and~$\SP(z)$ the Hilbert series of $A$ and $\Prim(A)$ defined as
$\SA(z) \eqdef \sum_{n=1}^{+\infty}\dim(A_n)z^n$ and
$\SP(z) \eqdef \sum_{n=1}^{+\infty}\dim(\Prim(A_n))z^n$.
The present work is based on two analogues of the Cartier-Milnor-Moore
theorems~\cite{foissy_2007} which we present now.
The first one is extracted from the proof of~{\cite[Proposition~6]{foissy_2011}}:
\begin{prop} \label{left_prod} Let $A$ be a bidendriform bialgebra and let
  $p_1\dots p_n \in \Prim(A)$. Then the map
  \begin{equation}
    p_1 \otimes p_2 \otimes \ldots \otimes p_n
    \mapsto p_1 \prec (p_2 \prec (\ldots \prec p_n)\ldots).
  \end{equation}
  is an isomorphism of co-algebras from $T^+(\Prim(A))$ (the non trivial part
  of the tensor algebra with deconcatenation as
  coproduct) to $A$.  As a consequence, taking a basis $(p_i)_{i\in I}$ of
  $\Prim(A)$, the
  family~$(p_{w_1} \prec (p_{w_2} \prec (\dots \prec p_{w_n})\dots))_w$
  where $w=w_1\dots w_n$ is a non empty word on $I$ defines a basis of $A$.  This
  implies the equality of Hilbert series $\SA=\SP/(1-\SP)$.
\end{prop}

One can further analyze $\Prim(A)$ using the so-called \emph{totally primitive}
elements of $A$ defined as
$\TPrim(A) = \Ker(\Delta_\prec) \cap \Ker(\Delta_\succ)$ and
$\ST(z) = \sum_{n=1}^{+\infty}\dim(\TPrim(A_n))z^n$.  Recall that a brace algebra is a
$K$-vector space $A$ together with an $n$-multilinear operation denoted as
$\langle\ldots\rangle$ for all $n ≥ 2$ which satisfies certain relations (see \cite{ronco_2000} for
details).
\begin{theorem}[{\cite[Theorem~5]{foissy_2011}}] \label{brace} Let $A$ be a bidendriform
  bialgebra. Then $\Prim(A)$ is freely generated as a brace algebra by
  $\TPrim(A)$ with brackets given by
  \begin{multline*}
    \langle p_1, \dots, p_{n-1}; p_n\rangle \eqdef \sum_{i=0}^{n-1}\ (-1)^{n-1-i}\\
    (p_1 \prec (p_2 \prec (\cdots \prec p_i)\cdots)) \succ p_n \prec ((\cdots(p_{i+1} \succ p_{i+2}) \succ \cdots) \succ p_{n-1}).
  \end{multline*}
\end{theorem}

A basis of $\Prim(A)$ is described by ordered trees that are decorated with
elements of $\TPrim(A)$ where $p_n$ is the root and $p_1, \ldots, p_{n-1}$ are the
children (see \cite{ronco_2000,chapoton_2002,foissy_2011}). This is reflected on
their Hilbert series as~\cite[{Corollary~37}]{foissy_2007}:
$\ST = \SA / (1+\SA)^2$ or equivalently $\SP = \ST(1 + \SA)$.

Using \cref{left_prod,brace} together with a dimension argument, one can show the
following theorem:
\begin{theorem} [{\cite[Theorem~2]{foissy_2011}}] \label{prim_tot}
Let $A$ be a bidendriform bialgebra. Then $A$ is freely generated as a dendriform
  algebra by $\TPrim(A)$.
\end{theorem}

\subsection{The Hopf algebra of word-quasisymmetric functions~\texorpdfstring{$\WQSym$}{WQSym}}

The algebra $\WQSym$ is a Hopf algebra whose bases are indexed by
ordered set partitions or equivalently surjections or packed words. In
this paper, we use the latter which we define now.
\begin{defi} \label{packed_word} A word over the alphabet~$\NN_{>0}$ is
  \emph{packed} if all the letters from $1$ to its maximum $m$ appears at least
  once. By convention, the empty word~$\epsilon$ is packed.  For $n\in\NN$, we denote by
  $\PW_n$ the set of all packed words of length (also called size) $n$ and
  $\PW = \bigsqcup_{n\in\NN}\PW_n$ the set of all packed words.
\end{defi}

\begin{defi} The packed word~$u\eqdef pack(w)$ associated with a word over the
  alphabet~$\NN_{>0}$ is obtained by the following process: if
  $b_1< b_2< \dots < b_r$ are the letters occurring in $w$, then $u$ is the
  image of $w$ by the homomorphism $b_i\mapsto i$.
\end{defi}
A word $u$ is packed if and only if $pack(u) = u$.
\begin{ex}
  The word $4152142$ is not packed because the letter $3$ does not appear while
  the maximum letter is $5 > 3$. Meanwhile $pack(4152142) = 3142132$ is a packed
  word. Here are all packed words of size $1$, $2$ and $3$ in lexicographic
  order:
  \[
    1,\quad 11\ 12\ 21,\quad
    111\ 112\ 121\ 122\ 123\ 132\ 211\ 212\ 213\ 221\ 231\ 312\ 321
  \]
\end{ex}
We will use the following notations and operations on words over the alphabet
$\NN_{>0}$: First, $\max(w)$ is the maximum letter of the word $w$ with the
convention that $\max(\epsilon)=0$. Then $|w|$ is the length (or size) of the word
$w$. The concatenation of the two words $u$ and $v$ is denoted as $u\cdot v$.
Moreover, $u\gcdot v$ (resp. $u\dcdot v$) is the left-shifted
(resp. right-shifted) concatenation of the two words where all the letters of
the left (resp. right) word are shifted by the maximum of the right (resp. left)
word: $1121\gcdot 3112 = 44543112$ and $1121\dcdot 3112 = 11215334$.

Finally $u\shuffle v$ is the shuffle product of the two words. It is
recursively defined as $u \shuffle \epsilon = \epsilon \shuffle u = u$ and
\begin{equation} \label{eq:shuffle}
  ua \shuffle vb = (u \shuffle vb)\cdot a + (ua \shuffle v)\cdot b
\end{equation}
where $u$ and $v$ are words and $a$ and $b$ are letters. Analogously to the
shifted concatenation, one can define the right shifted-shuffle $u\cshuffle v$
where all the letters of the right word $v$ are shifted by the maximum of the
left word $u$.

\begin{ex}
  $12\cshuffle \red{1}\red{1} =
  12\shuffle \red{3}\red{3} =
  12\red{3}\red{3} + 1\red{3}2\red{3} +
  1\red{3}\red{3}2 + \red{3}12\red{3} + \red{3}1\red{3}2 + \red{3}\red{3}12$.
\end{ex}

\paragraph*{Bidendriform bialgebra structure}

Novelli-Thibon~\cite{novthi_2006} proved that $\WQSym$ is a bidendriform
bialgebra. Their products and coproducts involve overlapping-shuffle. However it
will be easier for us to chose, among the various bases known in the literature
\cite{Hivert_thesis,berzag,novthi_2006,Vargas_thesis} a basis where the shuffle
are non-overlapping. Therefore, we take the dual basis denoted
$(\RR_w)_{w\in\PW}$ of \cite[Equation~23]{berzag}, using the classical bijection
between ordered partitions and packed words and redefine the bidendriform
structure. The Hopf algebra product and reduced coproduct are respectively
recovered as the sum of the half products (see~\cref{half_prod}) and half
coproducts (see~\cref{left_coprod,right_coprod}).

The recursive definition of the shuffle product (see \cref{eq:shuffle}) contains two
summands. We define them respectively as $\prec$ and $\succ$:
\begin{equation}
  ua \prec vb \eqdef (u \shuffle vb)\cdot a,
  \qquad\text{and}\qquad
  ua \succ vb \eqdef (ua \shuffle v)\cdot b.
\end{equation}
We define $\prec$, $\succ$, $\Delta_\prec$ and $\Delta_\succ$ on
$(\WQSym)_+ = Vect(\RR_u \mid u \in \PW_n, n \geq 1)$ in the following way:
for all $u = u_1\cdots u_n\in \PW_{n\geq1}$ and $v \in \PW_{m\geq1}$,
\begin{equation} \label{half_prod}
  \RR_u \prec \RR_v \eqdef \sum_{w \in u \prec v} \RR_w,
  \qquad\text{and}\qquad
  \RR_u \succ \RR_v \eqdef \sum_{w \in u \succ v} \RR_w.
\end{equation}
\begin{align} 
  \label{left_coprod}  \Delta_\prec(\RR_u) &\eqdef \sum_{\stackrel{i=k}{\stackrel{\{u_1, \dots, u_i\} \cap \{u_{i+1}, \dots, u_n\} = \emptyset}{u_k = \max(u)}}}^{n-1}
                                    \RR_{pack(u_1\cdots u_i)} \otimes \RR_{pack(u_{i+1}\cdots u_n)},\\
  \label{right_coprod}  \Delta_\succ(\RR_u) &\eqdef \sum_{\stackrel{i=1}{\stackrel{\{u_1, \dots, u_i\} \cap \{u_{i+1}, \dots, u_n\} = \emptyset}{u_k = \max(u)}}}^{k-1}
               \RR_{pack(u_1\cdots u_i)} \otimes \RR_{pack(u_{i+1}\cdots u_n)}.
\end{align}
\begin{ex}
  \begin{align*}
    \RR_{211} \prec \RR_{\red{12}} &= \RR_{21\red{3}\red{4}1} + \RR_{2\red{3}1\red{4}1} + \RR_{2\red{3}\red{4}11} + 
                                 \RR_{\red{3}21\red{4}1} + \RR_{\red{3}2\red{4}11} + \RR_{\red{3}\red{4}211},\\
    \RR_{221} \succ \RR_{\red{12}} &= \RR_{211\red{3}\red{4}} + \RR_{21\red{3}1\red{4}} + \RR_{2\red{3}11\red{4}}
                                 + \RR_{\red{3}211\red{4}},\\
    \Delta_\prec(\RR_{2125334}) &= \RR_{2123} \otimes \RR_{112} + \RR_{212433} \otimes \RR_{1},\\
    \Delta_\succ(\RR_{2125334}) &= \RR_{212} \otimes \RR_{3112}. 
  \end{align*}
\end{ex}

\begin{theorem} {\cite[{Theorem~2.5}]{novthi_2006}}
  $(\WQSym)_+, \prec, \succ, \Delta_\prec, \Delta_\succ)$ is a bidendriform bialgebra.
\end{theorem}
From now on $\Prim(\WQSym)$ and $\TPrim(\WQSym)$ are respectively abbreviated
to $\Prim$ and $\TPrim$. Moreover, we denote homogeneous components using
indices as in $\Prim_n$.
We give the first values of the dimensions~$a_n\eqdef\dim(\WQSym_n)$, $p_n\eqdef\dim(\Prim_n)$ and
$t_n\eqdef\dim(\TPrim_n)$:
\[
\begin{tabular}{|c|c|c|c|c|c|c|c|c|c||l|}
  \hline
  $n$ & 1 & 2 & 3 & 4 & 5 & 6 & 7 & 8 & 9 & OEIS \\
  \hline
  $a_n$ & 1 & 3 & 13 & 75 & 541 & 4 683 & 47 293 & 545 835 & 7 087 261 & A000670 \\
  \hline
  $p_n$ & 1 & 2 & 8 & 48 & 368 & 3 376 & 35 824 & 430 512 & 5 773 936 & A095989 \\
  \hline
  $t_n$ & 1 & 1 & 4 & 28 & 240 & 2 384 & 26 832 & 337 168 & 4 680 272 &  \\
  \hline
\end{tabular} 
\]


\section{Decorated forests}\label{sect2_pw_trees}

In this section we will generalize the construction of \cite{foissy_2011}. We
will start by defining some forests called biplane that are labeled by certain
lists of integers. We will construct a bijection between these packed forests
and packed words thanks to a decomposition of packed words through global
descents and removal of maximums. The recursive structure of forests can then be
understood as a chaining of operations generating the elements of
$\WQSym$. This will allow us to construct a basis of $\TPrim$ by characterizing
a subfamily of biplane trees.

\subsection{Decompositions of packed words}

\begin{defi} \label{global_descent} A \textbf{global descent} of a packed word
  $w$ is a position $c$ such that all the letters before or at position $c$ are
  greater than all letters after position $c$.
\end{defi}

\begin{ex}
  The global descents of $w = 54664312$ are the positions 5 and 6. Indeed, all letters of $54664$
  are greater than the letters of $312$ and this is also true for $546643$ and $12$.
\end{ex}

\begin{defi} \label{irreducible} A packed word $w$ is \textbf{irreducible} if it has no global
  descent.
\end{defi}

\begin{lem} \label{gd_fact} Each word~$w$ admits a unique factorization as
  $w = w_1\gcdot w_2\gcdot\dots\gcdot w_k$ such that $w_i$ is irreducible for all $i$.
\end{lem}

\begin{ex}
  The global descent decomposition of $54664312$ is $21331\gcdot 1\gcdot 12$.
  The word $n\cdot n-1\cdot...\cdot 1$ has $1\gcdot 1\gcdot ...\gcdot 1$ as global descent decomposition.
\end{ex}

\begin{defi} \label{phi} Fix $n \in \NN$ and $w \in \PW_n$. We write $m \eqdef \max(w) + 1$. For any
  $p>0$ and any subset~$I \subseteq [1,\dots, n+p]$ of cardinality~$p$,
  $\phi_I(w)=u_1\dots u_{n+p}$ is the packed word of length $n+p$ obtained by inserting $p$
  occurences of the letter $m$ in $w$ so that they end up in positions $i \in I$. In other words
  $u_i = m$ if $i\in I$ and $w$ is obtained from $\phi_I(w)$ by removing all occurrences of $m$.
\end{defi}

\begin{ex} $\phi_{2,4,7}(1232) = 1424324$ and $\phi_{1,2,3}(\epsilon)=111$.
\end{ex}

Let $\PW_n^I$ denote the set of packed words of size $n$ whose maximums are in positions
$i\in I$. This way $\phi_I(w)\in \PW_{n+p}^I$. The following lemma is immediate.

\begin{lem} \label{phi_bij} Let $n \in \NN$ and $p>0$, for any $I \subseteq [1,\dots, n+p]$,
  $\phi_I$ is a bijection from~$\PW_{n}$ to the~$\PW_{n+p}^I$.
  
  Moreover, for any $W\in\PW_\ell$ where~$\ell>0$ there exists a unique pair
  $(I,w)$ where $I\subseteq [1\dots \ell]$ and $w$ is packed, such that $W = \phi_I(w)$.
\end{lem}

\begin{defi} \label{ins} Let $u, v \in \PW$ with $v \neq \epsilon$. By~\cref{phi_bij}, there
  is a unique pair $(I, v')$ such that $v=\phi_I(v')$. We denote by $I'$ the set
  obtained by adding $|u|$ to the elements of~$I$. We define
  $u\ins v = \phi_{I'}(u\gcdot v')$. In other words, we remove the maximum letter
  of the right word, perform a left shifted concatenation and reinsert the
  removed letters as new maximums.
\end{defi}

\begin{ex}
  $2123\ins \red{3}\blue{22}\red{3}\blue{12} = \phi_{1+4, 4+4}(4345\blue{2212}) =
  4345\red{6}\blue{22}\red{6}\blue{12}$.
\end{ex}

\begin{lem} \label{phi_fact} Let~$w$ be an irreducible packed word. There exists
  a unique factorization of the form $w=u\ins\phi_I(v)$ which maximizes the size of
  $u$. In this factorization
  \begin{itemize}
  \item either $v=\epsilon$ and $I = [1, \dots, p]$ for some $p$;
  \item or the factorization $v=v_1\gcdot\dots\gcdot v_r$ of $v$ into
    irreducibles satisfies the inequalities $i_1 \leq |v_1|$, and
    $(\sum|v_i| + |I|) + 1 - i_p \leq |v_r|$.
  \end{itemize}
\end{lem}

\begin{ex} Here are some decompositions according to~\cref{phi_fact}:
  \begin{alignat*}{2}
    21331 &= 1\ins\phi_{2,3}(11) &
    1231 &= \epsilon \ins \phi_3(121)\\
    1233 &= 12 \ins \phi_{1,2}(\epsilon) & 
    111 &= \epsilon \ins \phi_{1,2,3}(\epsilon) \\
    543462161 &= (1 \gcdot 212) \ins \phi_{1,4}(1 \gcdot 11)
    \rlap{ $ = (3212) \ins \phi_{1,4}(211)$}
  \end{alignat*}
\end{ex}
  
\subsection{Forests from decomposed packed words}

We now apply recursively the decomposition of the former section to construct a
bijection between packed words and a certain kind of trees that we now define. 
\begin{defi} An unlabeled \textbf{biplane tree} is an ordered tree
  (sometimes also called a planar) whose children are organized in a
  pair of two (possibly empty) ordered forests, which we call the left
  and right forests.
\end{defi}

\begin{ex} The biplan trees
  $\scalebox{0.5}{{ \newcommand{\nodea}{\node[draw,circle] (a) {}
;}\newcommand{\nodeb}{\node[draw,circle] (b) {}
;}\newcommand{\nodec}{\node[draw,circle] (c) {}
;}\begin{tikzpicture}[auto]
\matrix[column sep=.3cm, row sep=.3cm,ampersand replacement=\&]{
         \&         \& \nodea  \\ 
 \nodeb  \& \nodec  \&         \\
};

\path[ultra thick, red] (a) edge (b) edge (c);
\end{tikzpicture}}}$,
  $\scalebox{0.5}{{ \newcommand{\nodea}{\node[draw,circle] (a) {}
;}\newcommand{\nodeb}{\node[draw,circle] (b) {}
;}\newcommand{\nodec}{\node[draw,circle] (c) {}
;}\begin{tikzpicture}[auto]
\matrix[column sep=.3cm, row sep=.3cm,ampersand replacement=\&]{
         \& \nodea  \&         \\ 
 \nodeb  \&         \& \nodec  \\
};

\path[ultra thick, red] (a) edge (b) edge (c);
\end{tikzpicture}}}$ and
  $\scalebox{0.5}{{ \newcommand{\nodea}{\node[draw,circle] (a) {}
;}\newcommand{\nodeb}{\node[draw,circle] (b) {}
;}\newcommand{\nodec}{\node[draw,circle] (c) {}
;}\begin{tikzpicture}[auto]
\matrix[column sep=.3cm, row sep=.3cm,ampersand replacement=\&]{
 \nodea  \&         \&         \\ 
         \& \nodeb  \& \nodec  \\
};

\path[ultra thick, red] (a) edge (b) edge (c);
\end{tikzpicture}}}$ are different.
  Indeed in the first case, the left forest contains two trees and the right
  forest is empty, in the second case both forests contain exactly one tree
  while in the third case we have the opposite of the first case. Here is an example
  of a bigger biplane tree where the root has two trees in both left and right
  forests \scalebox{0.5}{{ \newcommand{\nodea}{\node[draw,circle] (a) {}
;}\newcommand{\nodeb}{\node[draw,circle] (b) {}
;}\newcommand{\nodec}{\node[draw,circle] (c) {}
;}\newcommand{\noded}{\node[draw,circle] (d) {}
;}\newcommand{\nodee}{\node[draw,circle] (e) {}
;}\newcommand{\nodef}{\node[draw,circle] (f) {}
;}\newcommand{\nodeg}{\node[draw,circle] (g) {}
;}\newcommand{\nodeh}{\node[draw,circle] (h) {}
;}\newcommand{\nodei}{\node[draw,circle] (i) {}
;}\newcommand{\nodej}{\node[draw,circle] (j) {}
;}\begin{tikzpicture}[auto]
\matrix[column sep=.3cm, row sep=.3cm,ampersand replacement=\&]{
         \&         \&         \&         \&         \&         \&         \& \nodea  \&         \&         \&         \\ 
         \& \nodeb  \&         \&         \& \nodee  \&         \&         \&         \&         \& \nodeh  \& \nodej  \\ 
 \nodec  \&         \& \noded  \&         \&         \& \nodef  \& \nodeg  \&         \& \nodei  \&         \&         \\
};

\path[ultra thick, red] (b) edge (c) edge (d)
	(e) edge (f) edge (g)
	(h) edge (i)
	(a) edge (b) edge (e) edge (h) edge (j);
\end{tikzpicture}}}.
\end{ex}

\begin{rem}
  These biplan trees are counted by the sequence A006013 in OEIS.
\end{rem}

In our construction we will deal with labeled biplane trees where the
labels are sorted lists of positive integer. For a labeled biplane
tree, we denote by $\Node(x, f_\ell, f_r)$ the tree whose root is labeled
by $x$ and whose left (resp. right) forest is given by $f_\ell$
(resp. $f_r$). We also denote by~$[t_1, \dots t_k]$ a forest of $k$
trees.

\begin{ex}
  $\Node((1),\ [],[]) =$ \scalebox{0.5}{{ \newcommand{\nodea}{\node[draw,circle] (a) {$1$}
;}\begin{tikzpicture}[auto]
\matrix[column sep=.3cm, row sep=.3cm,ampersand replacement=\&]{
 \nodea  \\
};
\end{tikzpicture}}}, and
  $\Node((1,3),\ [],[\Node((1),\ [],[])]) =$ \scalebox{0.5}{{ \newcommand{\nodea}{\node[draw,circle] (a) {$1, 3$}
;}\newcommand{\nodeb}{\node[draw,circle] (b) {$1$}
;}\begin{tikzpicture}[auto]
\matrix[column sep=.3cm, row sep=.3cm,ampersand replacement=\&]{
         \& \nodea  \&         \\ 
         \&         \& \nodeb  \\
};

\path[ultra thick, red] (a) edge (b);
\end{tikzpicture}}}.
\end{ex}

We now apply recursively the decompositions of~\cref{gd_fact,phi_fact} to get an algorithm which
takes a packed word and returns a biplane forest where nodes are decorated by lists of integers:
\begin{defi} \label{def:construction}
The forest~$F(w)$~(resp. tree~$T(w)$) associated to a packed word (resp. a non empty
    irreducible packed word) $w$ are defined in a mutual recursive way as follows:
  \begin{itemize}
  \item $F(\epsilon) = []$ (empty forest),
  \item for any packed word $w$, let
    $w_1\gcdot w_2\gcdot\dots\gcdot w_k$ be the global descent decomposition
    of $w$, then $F(w) \eqdef [T(w_1), T(w_2), \dots, T(w_k)]$.
  \item for any non empty irreducible packed word $w$, we define
    $T(w) \eqdef \Node(I, F(u), F(v))$ where $w=u\ins\phi_I(v)$ and $u$ is of
    maximal length.
  \end{itemize}
\end{defi}

\begin{ex}
  Let $w = 876795343912$, the decomposition of~\cref{gd_fact} gives $w = w_1\gcdot w_2$ with
  $w_1 = 6545731217$ and $w_2 = 12$. Now, we decompose $w_1$ and $w_2$
  using~\cref{phi_fact} as
   \[
     w_1 = 3212 \ins \phi_{1,6}( 3121 )= (1 \gcdot 212) \ins \phi_{1,6} (1
     \gcdot 121),
     \quad\text{and}\quad
     w_2 = 1 \ins \phi_1(\epsilon).
   \]
   It gives the following forest:
  \begin{align*}
    F(876795343912) & = [T(6545731217),\ T(12)] \\
                    & =
                      \scalebox{0.6}{{ \newcommand{\nodea}{\node[draw,circle] (a) {$1, 6$}
;}\newcommand{\nodeb}{\node (b) {$F(3212)$}
;}\newcommand{\nodee}{\node (e) {$F(3121)$}
;}\begin{tikzpicture}[auto]
\matrix[column sep=.3cm, row sep=.3cm,ampersand replacement=\&]{
         \&         \&         \&         \& \nodea  \&        \\
 \nodeb  \&         \&         \&         \&         \& \nodee \\
};

\path[ultra thick, red] (a) edge (b) edge (e);
\end{tikzpicture}}}
                      \scalebox{0.6}{{ \newcommand{\nodea}{\node[draw,circle] (a) {$1$}
;}\newcommand{\nodeb}{\node (b) {$F(1)$}
;}\newcommand{\nodee}{\node (e) {$F(\epsilon)$}
;}\begin{tikzpicture}[auto]
\matrix[column sep=.3cm, row sep=.3cm,ampersand replacement=\&]{
         \& \nodea  \&         \\
 \nodeb  \&         \& \nodee  \\
};

\path[ultra thick, red] (a) edge (b) edge (e);
\end{tikzpicture}}
} \\ 
                    & =
                      \scalebox{0.5}{{ \newcommand{\nodea}{\node[draw,circle] (a) {$1, 6$}
;}\newcommand{\nodeb}{\node[draw,circle] (b) {$1$}
;}\newcommand{\nodec}{\node[draw,circle] (c) {$1, 3$}
;}\newcommand{\noded}{\node[draw,circle] (d) {$1$}
;}\newcommand{\nodee}{\node[draw,circle] (e) {$1$}
;}\newcommand{\nodef}{\node[draw,circle] (f) {$2$}
;}\newcommand{\nodeg}{\node[draw,circle] (g) {$1, 2$}
;}\begin{tikzpicture}[auto]
\matrix[column sep=.3cm, row sep=.3cm,ampersand replacement=\&]{
         \&         \&         \&         \& \nodea  \&         \&         \&         \&         \\ 
 \nodeb  \&         \& \nodec  \&         \&         \& \nodee  \&         \& \nodef  \&         \\ 
         \&         \&         \& \noded  \&         \&         \&         \&         \& \nodeg  \\
};

\path[ultra thick, red] (c) edge (d)
	(f) edge (g)
	(a) edge (b) edge (c) edge (e) edge (f);
\end{tikzpicture}}
{  \newcommand{\nodea}{\node[draw,circle] (a) {$1$}
;}\newcommand{\nodeb}{\node[draw,circle] (b) {$1$}
;}\begin{tikzpicture}[auto]
\matrix[column sep=.3cm, row sep=.3cm,ampersand replacement=\&]{
         \& \nodea  \\ 
 \nodeb  \&         \\
};

\path[ultra thick, red] (a) edge (b);
\end{tikzpicture}}}
                      =
                      \scalebox{0.5}{{ \newcommand{\nodea}{\node[draw,circle] (a) {$1, 6$}
;}\newcommand{\nodeb}{\node[draw,circle] (b) {}
;}\newcommand{\nodec}{\node[draw,circle] (c) {$1, 3$}
;}\newcommand{\noded}{\node[draw,circle] (d) {}
;}\newcommand{\nodee}{\node[draw,circle] (e) {}
;}\newcommand{\nodef}{\node[draw,circle] (f) {$2$}
;}\newcommand{\nodeg}{\node[draw,circle] (g) {$1, 2$}
;}\begin{tikzpicture}[auto]
\matrix[column sep=.3cm, row sep=.3cm,ampersand replacement=\&]{
         \&         \&         \&         \& \nodea  \&         \&         \&         \&         \\ 
 \nodeb  \&         \& \nodec  \&         \&         \& \nodee  \&         \& \nodef  \&         \\ 
         \&         \&         \& \noded  \&         \&         \&         \&         \& \nodeg  \\
};

\path[ultra thick, red] (c) edge (d)
	(f) edge (g)
	(a) edge (b) edge (c) edge (e) edge (f);
\end{tikzpicture}}
{  \newcommand{\nodea}{\node[draw,circle] (a) {}
;}\newcommand{\nodeb}{\node[draw,circle] (b) {}
;}\begin{tikzpicture}[auto]
\matrix[column sep=.3cm, row sep=.3cm,ampersand replacement=\&]{
         \& \nodea  \\ 
 \nodeb  \&         \\
};

\path[ultra thick, red] (a) edge (b);
\end{tikzpicture}}}.
  \end{align*}
  
  As we can see in this example, in order to simplify the notation, we do not
  write the label when it is the list $1$.
  We now characterize the trees obtained this way:
\end{ex}

\begin{defi} \label{forest_def} Let $t$ be a labeled biplane tree. We
  write $t = \Node(I, f_\ell, f_r)$ where ${I = [i_1, \ldots, i_p]}$,
  $f_\ell = [\ell_1, \ldots, \ell_g]$ and $f_r = [r_1, \ldots, r_d]$, which is depicted as follows:
  \begin{equation*}
  	t = \scalebox{0.8}{{ \newcommand{\nodea}{\node[draw,circle] (a) {$I$}
;}\newcommand{\nodeb}{\node (b) {$\ell_1$}
;}\newcommand{\nodebc}{\node (bc) {$\ldots$}
;}\newcommand{\nodec}{\node (c) {$\ell_g$}
;}\newcommand{\nodee}{\node (e) {$r_1$}
;}\newcommand{\nodeef}{\node (ef) {$\ldots$}
;}\newcommand{\nodef}{\node (f) {$r_d$}
;}\begin{tikzpicture}[auto]
\matrix[column sep=.3cm, row sep=.3cm,ampersand replacement=\&]{
      \&  \&  \& \nodea  \& \&    \&    \\ 
 \nodeb  \& \nodebc \& \nodec \&  \& \nodee \& \nodeef \& \nodef  \\
};

\path[ultra thick, red] (a) edge (b) edge (c) edge (e) edge (f);

\end{tikzpicture}}}.
  \end{equation*}
  
  The \emph{weight} of $t$ is recursively defined by
  $\omega(t) = p + \sum_{i=0}^{g}\omega(\ell_i) + \sum_{j=0}^{d}\omega(r_j)$. In particular, if
  $t$ is a leaf then $\omega(t) = p$. The \emph{right-weight} of $t$ is defined by
  ${\omega_r(t) = p + \sum_{j=0}^{d}\omega(r_j)}$. We say that $t$ is a \textbf{packed tree}
  if it satisfies:
  \begin{itemize}
  \item If there are no right children ($d=0$) then $i_k = k$ for all $k$.
  \Item Otherwise:
    \begin{equation} \label{eq:forest_def}
      \left\{
        \begin{array}{rcl}
          1 \leq &i_1 &\leq \omega(r_1),\\
          1 \leq &\omega_r(t) + 1 - i_p &\le \omega(r_d).
        \end{array}
      \right.
    \end{equation}
  \item $f_\ell$ and $f_r$ are packed forests (\emph{i.e.} lists of packed trees).
  \end{itemize}
\end{defi}

For the rest of the article we denote the set of packed forests of weight $n$ by $\PForest_n$
and the set of packed trees of weight $n$ by $\PTree_n$. We also denote by $\PForest_n^I$ the
set of packed forests that satisfy~\cref{eq:forest_def} for $I = (i_1, \ldots, i_p)$ with
$0 <i_1< \ldots< i_p$. Finally we denote the set of packed trees of weight
$n$ with an empty left forest  by $\PParticular_n$.

There is a unique packed forest of weight 1, namely \scalebox{0.5}{{ \newcommand{\nodea}{\node[draw,circle] (a) {}
;}\begin{tikzpicture}[auto]
\matrix[column sep=.3cm, row sep=.3cm,ampersand replacement=\&]{
 \nodea  \\
};
\end{tikzpicture}}},
here are the packed forests of weight 2:\scalebox{0.5}{{ \newcommand{\nodea}{\node[draw,circle] (a) {}
;}\newcommand{\nodeb}{\node[draw,circle] (b) {}
;}\begin{tikzpicture}[auto]
\matrix[column sep=.3cm, row sep=.3cm,ampersand replacement=\&]{
         \& \nodea  \\ 
 \nodeb  \&         \\
};

\path[ultra thick, red] (a) edge (b);
\end{tikzpicture}}},
    \scalebox{0.5}{{ \newcommand{\nodea}{\node[draw,circle] (a) {}
;}\begin{tikzpicture}[auto]
\matrix[column sep=.3cm, row sep=.3cm,ampersand replacement=\&]{
 \nodea  \\
};
\end{tikzpicture}}
{  \newcommand{\nodea}{\node[draw,circle] (a) {}
;}\begin{tikzpicture}[auto]
\matrix[column sep=.3cm, row sep=.3cm,ampersand replacement=\&]{
 \nodea  \\
};
\end{tikzpicture}}},
    \scalebox{0.5}{{ \newcommand{\nodea}{\node[draw,circle] (a) {$1, 2$}
;}\begin{tikzpicture}[auto]
\matrix[column sep=.3cm, row sep=.3cm,ampersand replacement=\&]{
 \nodea  \\
};
\end{tikzpicture}}}.
    We show below the packed forests of weight 3:

\noindent
    \scalebox{0.4}{{ \newcommand{\nodea}{\node[draw,circle] (a) {}
;}\newcommand{\nodeb}{\node[draw,circle] (b) {}
;}\newcommand{\nodec}{\node[draw,circle] (c) {}
;}\begin{tikzpicture}[auto]
\matrix[column sep=.3cm, row sep=.3cm,ampersand replacement=\&]{
         \&         \& \nodea  \\ 
         \& \nodeb  \&         \\ 
 \nodec  \&         \&         \\
};

\path[ultra thick, red] (b) edge (c)
	(a) edge (b);
\end{tikzpicture}}},
    \scalebox{0.4}{{ \newcommand{\nodea}{\node[draw,circle] (a) {$2$}
;}\newcommand{\nodeb}{\node[draw,circle] (b) {}
;}\newcommand{\nodec}{\node[draw,circle] (c) {}
;}\begin{tikzpicture}[auto]
\matrix[column sep=.3cm, row sep=.3cm,ampersand replacement=\&]{
         \& \nodea  \&         \&         \\ 
         \&         \&         \& \nodeb  \\ 
         \&         \& \nodec  \&         \\
};

\path[ultra thick, red] (b) edge (c)
	(a) edge (b);
\end{tikzpicture}}},
    \scalebox{0.4}{{ \newcommand{\nodea}{\node[draw,circle] (a) {}
;}\newcommand{\nodeb}{\node[draw,circle] (b) {}
;}\newcommand{\nodec}{\node[draw,circle] (c) {}
;}\begin{tikzpicture}[auto]
\matrix[column sep=.3cm, row sep=.3cm,ampersand replacement=\&]{
         \&         \& \nodea  \\ 
 \nodeb  \& \nodec  \&         \\
};

\path[ultra thick, red] (a) edge (b) edge (c);
\end{tikzpicture}}},
    \scalebox{0.4}{{ \newcommand{\nodea}{\node[draw,circle] (a) {}
;}\newcommand{\nodeb}{\node[draw,circle] (b) {}
;}\begin{tikzpicture}[auto]
\matrix[column sep=.3cm, row sep=.3cm,ampersand replacement=\&]{
         \& \nodea  \\ 
 \nodeb  \&         \\
};

\path[ultra thick, red] (a) edge (b);
\end{tikzpicture}}
{  \newcommand{\nodea}{\node[draw,circle] (a) {}
;}\begin{tikzpicture}[auto]
\matrix[column sep=.3cm, row sep=.3cm,ampersand replacement=\&]{
 \nodea  \\
};
\end{tikzpicture}}},
    \scalebox{0.4}{{ \newcommand{\nodea}{\node[draw,circle] (a) {}
;}\begin{tikzpicture}[auto]
\matrix[column sep=.3cm, row sep=.3cm,ampersand replacement=\&]{
 \nodea  \\
};
\end{tikzpicture}}
{  \newcommand{\nodea}{\node[draw,circle] (a) {}
;}\newcommand{\nodeb}{\node[draw,circle] (b) {}
;}\begin{tikzpicture}[auto]
\matrix[column sep=.3cm, row sep=.3cm,ampersand replacement=\&]{
         \& \nodea  \\ 
 \nodeb  \&         \\
};

\path[ultra thick, red] (a) edge (b);
\end{tikzpicture}}},
    \scalebox{0.4}{{ \newcommand{\nodea}{\node[draw,circle] (a) {}
;}\begin{tikzpicture}[auto]
\matrix[column sep=.3cm, row sep=.3cm,ampersand replacement=\&]{
 \nodea  \\
};
\end{tikzpicture}}
{  \newcommand{\nodea}{\node[draw,circle] (a) {}
;}\begin{tikzpicture}[auto]
\matrix[column sep=.3cm, row sep=.3cm,ampersand replacement=\&]{
 \nodea  \\
};
\end{tikzpicture}}
{  \newcommand{\nodea}{\node[draw,circle] (a) {}
;}\begin{tikzpicture}[auto]
\matrix[column sep=.3cm, row sep=.3cm,ampersand replacement=\&]{
 \nodea  \\
};
\end{tikzpicture}}},
    \scalebox{0.4}{{ \newcommand{\nodea}{\node[draw,circle] (a) {$1, 2$}
;}\newcommand{\nodeb}{\node[draw,circle] (b) {}
;}\begin{tikzpicture}[auto]
\matrix[column sep=.3cm, row sep=.3cm,ampersand replacement=\&]{
         \& \nodea  \\ 
 \nodeb  \&         \\
};

\path[ultra thick, red] (a) edge (b);
\end{tikzpicture}}},
    \scalebox{0.4}{{ \newcommand{\nodea}{\node[draw,circle] (a) {$1, 3$}
;}\newcommand{\nodeb}{\node[draw,circle] (b) {}
;}\begin{tikzpicture}[auto]
\matrix[column sep=.3cm, row sep=.3cm,ampersand replacement=\&]{
         \& \nodea  \&         \\ 
         \&         \& \nodeb  \\
};

\path[ultra thick, red] (a) edge (b);
\end{tikzpicture}}},
    \scalebox{0.4}{{ \newcommand{\nodea}{\node[draw,circle] (a) {$1, 2$}
;}\begin{tikzpicture}[auto]
\matrix[column sep=.3cm, row sep=.3cm,ampersand replacement=\&]{
 \nodea  \\
};
\end{tikzpicture}}
{  \newcommand{\nodea}{\node[draw,circle] (a) {}
;}\begin{tikzpicture}[auto]
\matrix[column sep=.3cm, row sep=.3cm,ampersand replacement=\&]{
 \nodea  \\
};
\end{tikzpicture}}},
    \scalebox{0.4}{{ \newcommand{\nodea}{\node[draw,circle] (a) {}
;}\newcommand{\nodeb}{\node[draw,circle] (b) {$1, 2$}
;}\begin{tikzpicture}[auto]
\matrix[column sep=.3cm, row sep=.3cm,ampersand replacement=\&]{
         \& \nodea  \\ 
 \nodeb  \&         \\
};

\path[ultra thick, red] (a) edge (b);
\end{tikzpicture}}},
    \scalebox{0.4}{{ \newcommand{\nodea}{\node[draw,circle] (a) {$2$}
;}\newcommand{\nodeb}{\node[draw,circle] (b) {$1, 2$}
;}\begin{tikzpicture}[auto]
\matrix[column sep=.3cm, row sep=.3cm,ampersand replacement=\&]{
         \& \nodea  \&         \\ 
         \&         \& \nodeb  \\
};

\path[ultra thick, red] (a) edge (b);
\end{tikzpicture}}},
    \scalebox{0.4}{{ \newcommand{\nodea}{\node[draw,circle] (a) {}
;}\begin{tikzpicture}[auto]
\matrix[column sep=.3cm, row sep=.3cm,ampersand replacement=\&]{
 \nodea  \\
};
\end{tikzpicture}}
{  \newcommand{\nodea}{\node[draw,circle] (a) {$1, 2$}
;}\begin{tikzpicture}[auto]
\matrix[column sep=.3cm, row sep=.3cm,ampersand replacement=\&]{
 \nodea  \\
};
\end{tikzpicture}}},
    \scalebox{0.4}{{ \newcommand{\nodea}{\node[draw,circle] (a) {$1, 2, 3$}
;}\begin{tikzpicture}[auto]
\matrix[column sep=.3cm, row sep=.3cm,ampersand replacement=\&]{
 \nodea  \\
};
\end{tikzpicture}}}.

The following theorem is a generalisation of the construction of \cite{foissy_2011} for $\FQSym$
and permutations to $\WQSym$ and packed words.
\begin{theorem} \label{thm:bij} For all $n\in\NN$ we have the three following
  equalities : 
  \begin{equation*}
    \dim(\WQSym_n) = \#\PForest_n
    \quad\text{and}\quad
    \dim(\Prim_n) = \#\PTree_n
    \quad\text{and}\quad
    \dim(\TPrim_n) = \#\PParticular_n
  \end{equation*}
\end{theorem}

\begin{proof}
  The construction of \cref{def:construction} defines a bijection from
  packed words to packed forests which restrict to a bijection from
  irreducible packed words to packed trees. Thanks to \cref{left_prod}
  this proves the first two equalities. Recall that a basis of
  primitive elements is given by~\cref{brace} as ordered trees
  decorated by totally primitive elements. If we consider the label
  $I$ together with the right forest of each node as decoration, we
  get that packed trees are in bijection with planar trees decorated
  by $I$ and an element of $\PForest^I$ that is an element
  of~$\PParticular$.
\end{proof}


\section{A basis for totally primitive elements}
\label{sect3_Pbasis}
In this section we construct a basis of primitive and totally primitive
elements of $\WQSym$. Thanks to~\cref{thm:bij} we now have the combinatorial objects to
index those basis. To have the linear independency, we need to show that the
decomposition through maximum is compatible with the algebraic structure.

\subsection{Decomposition through maximums and totally primitive elements}

\begin{defi} \label{Phi} Let $I = (i_1, \ldots, i_p)$ with $0 <i_1< \ldots< i_p$. We define the linear map
  $\Phi_I:\WQSym \to \WQSym$ as follows: for all $n\in\NN$ and $w = w_1\cdot w_2 \cdots w_n \in\PW_n$,
  \begin{equation}
    \Phi_I(\RR_w) \eqdef \left\{
      \begin{array}{rl}
        \RR_{\phi_I(w)} & \text{if } i_p \leq n + p,\\
        0 & \text{if } i_p > n + p\,.
      \end{array}\right.
  \end{equation}
\end{defi}

\begin{defi} \label{tau} Let $I = (i_1, \ldots, i_p)$ with $0 <i_1< \ldots< i_p$. We define the projector
  $\tau_I: \WQSym \to \WQSym$ as follows: for all $n\in\NN$ and $w = w_1\cdot w_2 \cdots w_n \in\PW_n$,
  \begin{equation}
    \tau_I(\RR_w) \eqdef \left\{
      \begin{array}{rl}
        \RR_w & \text{if } w_i = \max(w) \text{ if and only if } i \in I,\\
        0 & \text{else.}
      \end{array}\right.
  \end{equation}
  These are orthogonal projectors in the sense that $\tau_I^2 = \tau _I$ and $\tau_I \circ \tau_J = 0$ ($I \neq J$).
\end{defi}

\begin{lem}\label{im_eq} For any $I$, we have $Im(\Phi_I) = Im(\tau_I)$ where
  $Im(f)$ denotes the image of $f$.
\end{lem}

\begin{lem} \label{tot_stable} For any~$I$, the projection by $\tau_I$ of a
  totally primitive element is still a totally primitive element, so that
  $\tau_I(\TPrim) = Im(\tau_I) \bigcap \TPrim$. Moreover,
  \begin{equation}\label{eq:direct_sum}
    \TPrim = \bigoplus_I Im(\tau_I) \cap \TPrim\,.
  \end{equation}
\end{lem}

\begin{proof}
  Let $w$ a packed word. We have
  $\Delta_\prec(\tau_I(\RR_w)) = (\tau_I \otimes Id) \circ \Delta_\prec(\RR_w)$ by definition of
  $\tau$ and $\Delta_\prec$. By linearity, for all $p \in \TPrim,$ we have
  $\Delta_\prec(\tau_I(p)) = (\tau_I \otimes Id) \circ \Delta_\prec(p) = 0$. The same argument works on the right so
  that~$\tau_I(p)\in\TPrim$. Morevover $\tau_I$ are orthogonal projectors so
  $\TPrim = \bigoplus_I \tau_I(\TPrim) = \bigoplus_I Im(\tau_I)\cap\TPrim$.
\end{proof}

\subsection{The new basis ~\texorpdfstring{$\PP$}{PP}}

\begin{defi} \label{P} Let $t_1, \ldots, t_k \in \PTree$ and $f_l = [\ell_1, \ldots, \ell_g], f_r \in \PForest$,
  \begin{align}
    \label{P1}  \PP_{\scalebox{0.5}{{ \newcommand{\nodea}{\node[draw,circle] (a) {}
;}\begin{tikzpicture}[auto]
\matrix[column sep=.3cm, row sep=.3cm,ampersand replacement=\&]{
 \nodea  \\
};
\end{tikzpicture}}}}  & \eqdef \RR_1,\\
    \label{PF} \PP_{t_1, \ldots, t_k} & \eqdef \PP_{t_k} \prec (\PP_{t_{k-1}} \prec (\ldots  \prec \PP_{t_1})\ldots),\\
    \label{PP} \PP_{\Node(I, [], f_r)} & \eqdef \Phi_{I}(\PP_{f_r}),\\
    \label{PT} \PP_{\Node(I, f_l = [\ell_1, \ldots, \ell_g], f_r)} & \eqdef \langle\PP_{l_1}, \PP_{l_2}, \ldots, \PP_{l_g}; \Phi_{I}(\PP_{f_r}) \rangle.
  \end{align}
\end{defi}

\begin{ex}
  \begin{align*}
    \PP_{\scalebox{0.5}{{ \newcommand{\nodea}{\node[draw,circle] (a) {$1, 3$}
;}\newcommand{\nodeb}{\node[draw,circle] (b) {}
;}\begin{tikzpicture}[auto]
\matrix[column sep=.3cm, row sep=.3cm,ampersand replacement=\&]{
         \& \nodea  \&         \\ 
         \&         \& \nodeb  \\
};

\path[ultra thick, red] (a) edge (b);
\end{tikzpicture}}
{  \newcommand{\nodea}{\node[draw,circle] (a) {}
;}\newcommand{\nodeb}{\node[draw,circle] (b) {}
;}\begin{tikzpicture}[auto]
\matrix[column sep=.3cm, row sep=.3cm,ampersand replacement=\&]{
         \& \nodea  \\ 
 \nodeb  \&         \\
};

\path[ultra thick, red] (a) edge (b);
\end{tikzpicture}}}} &= \PP_{\scalebox{0.5}{{ \newcommand{\nodea}{\node[draw,circle] (a) {}
;}\newcommand{\nodeb}{\node[draw,circle] (b) {}
;}\begin{tikzpicture}[auto]
\matrix[column sep=.3cm, row sep=.3cm,ampersand replacement=\&]{
         \& \nodea  \\ 
 \nodeb  \&         \\
};

\path[ultra thick, red] (a) edge (b);
\end{tikzpicture}}}} \prec \PP_{\scalebox{0.5}{{ \newcommand{\nodea}{\node[draw,circle] (a) {$1, 3$}
;}\newcommand{\nodeb}{\node[draw,circle] (b) {}
;}\begin{tikzpicture}[auto]
\matrix[column sep=.3cm, row sep=.3cm,ampersand replacement=\&]{
         \& \nodea  \&         \\ 
         \&         \& \nodeb  \\
};

\path[ultra thick, red] (a) edge (b);
\end{tikzpicture}}}} = (\PP_{\scalebox{0.5}{{ \newcommand{\nodea}{\node[draw,circle] (a) {}
;}\begin{tikzpicture}[auto]
\matrix[column sep=.3cm, row sep=.3cm,ampersand replacement=\&]{
 \nodea  \\
};
\end{tikzpicture}}}} \prec \PP_{\scalebox{0.5}{{ \newcommand{\nodea}{\node[draw,circle] (a) {}
;}\begin{tikzpicture}[auto]
\matrix[column sep=.3cm, row sep=.3cm,ampersand replacement=\&]{
 \nodea  \\
};
\end{tikzpicture}}}} - \PP_{\scalebox{0.5}{{ \newcommand{\nodea}{\node[draw,circle] (a) {}
;}\begin{tikzpicture}[auto]
\matrix[column sep=.3cm, row sep=.3cm,ampersand replacement=\&]{
 \nodea  \\
};
\end{tikzpicture}}}} \succ \PP_{\scalebox{0.5}{{ \newcommand{\nodea}{\node[draw,circle] (a) {}
;}\begin{tikzpicture}[auto]
\matrix[column sep=.3cm, row sep=.3cm,ampersand replacement=\&]{
 \nodea  \\
};
\end{tikzpicture}}}}) \prec \Phi_{1,3}(\PP_{\scalebox{0.5}{{ \newcommand{\nodea}{\node[draw,circle] (a) {}
;}\begin{tikzpicture}[auto]
\matrix[column sep=.3cm, row sep=.3cm,ampersand replacement=\&]{
 \nodea  \\
};
\end{tikzpicture}}}})\\
                                                               &= \RR_{14342} + \RR_{41342} + \RR_{43142} + \RR_{43412} - \RR_{24341} - \RR_{42341} - \RR_{43241} - \RR_{43421}
  \end{align*}
\end{ex}

\begin{theorem} \label{thm:P} For all $n \in \NN_{>0}$
  \begin{multicols}{2}
  \begin{enumerate}
  \item \label{PForest} $(\PP_f)_{f\in\PForest_n}$ is a basis of $\WQSym_n$,
  \item \label{PTree} $(\PP_t)_{t\in\PTree_n}$ is a basis of $\Prim_n$,
  \item \label{PParticular} $(\PP_t)_{t\in\PParticular_n}$ is a basis of $\TPrim_n$.
  \end{enumerate}
  \end{multicols}
\end{theorem}

\begin{proof} As $\dim(\WQSym_1) = \dim(\Prim_1) = \dim(\TPrim_1) = 1$ the base
  case is trivial. We argue in a mutually recursive way: by \cref{left_prod},
  \cref{PTree} up to degree $n$ implies \cref{PForest} up to degree
  $n$. Similarly, \cref{brace} shows that \cref{PParticular} up to degree $n$
  implies \cref{PTree} up to degree $n$. By induction it is sufficient to show
  that \cref{PForest,PTree} up to degree $n-1$ implies \cref{PParticular} for
  $n$. 

  For all $k \in \NN$, let $\pi_k$ be the projector on the homogeneous component of
  degree $k$ of $\WQSym$. We define $\pi_{<k} \eqdef \sum_{i=0}^{k-1}\pi_i$. Fix
  $I$ of length $p$. In the coproduct $\Delta_\prec(\Phi_I(\RR_u))$ all the maximums must be
  in the left tensor factor, which therefore must be at least of degree $i_p$.
  By linearity, this can be used to prove that
  $\Delta_\prec(\Phi_I(\PP_{f})) = 0$. A similar reasoning applies to
  $\Delta_\succ(\Phi_I(x))$, so that thanks to the conditions of~\cref{eq:forest_def},
  $\PP_t$ with $t \in \PParticular_n$ is totally primitive.

  In other words, if we define
  $\Prim_n(i,j) \eqdef \Ker(\pi_{<i} \otimes \pi_{<j}) \circ \tilde{\Delta}$ we have proved that
  the image of the restriction to $\Prim_n(i_1, n + 1 - i_p)$ of $\Phi_I$ is
  included in $Im(\tau_I) \cap \TPrim_n$. By \cref{left_prod},
  $\{\PP_f\mid f\in\PForest_{n-p}^I\}$ is a basis of
  $\Prim_{n-p}(i_1, n+1-i_p)$. Since $\Phi_I$ is injective on $\WQSym_{n-p}$ then
  $\{\Phi_I(\PP_f)\mid f\in\PForest_{n-p}^I\}$ are linearly independent. Then by
  \cref{tot_stable} $\{\PP_t\mid t\in\PParticular_n\}$ are linearly independent. By
  \cref{thm:bij} it is a basis of $\TPrim_n$.
\end{proof}


\section*{Conclusion}

Our next obvious step is to adapt the decomposition through maximum
of~\cref{Phi,tau} to the dual in order to get an explicit bidendriform
isomorphism from $\WQSym$ to its dual. Another generalization which should be
easy is to do the same for the Hopf algebra $\PQSym$ of parking
functions. Indeed $\PQSym$ is a bidendriform bialgebra~\cite{novthi_2006}, and
thus self-dual, but no explicit isomorphism with the dual is known.

\paragraph*{Acknowledgments}
The computation and tests needed for this research were done using the
open-source mathematical software \textsc{Sage} and its combinatorics
features developed by the \textsc{Sage-combinat} community. I
particularly thank F. Hivert and V. Pons for all comments on the
writing.


\bibliography{biblio}

\begin{thebibliography}{10}

\bibitem{NCSF1}
I.~Gelfand, D.~Krob, A.~Lascoux, B.~Leclerc, V.~Retakh, and J.~Thibon,
  ``{Noncommutative Symmetrical Functions},'' {\em Adv. in Math.}, vol.~112,
  no.~2, pp.~218 -- 348, 1995.

\bibitem{LodRon_PBT}
J.-L. Loday and M.~O. Ronco, ``{Hopf Algebra of the Planar Binary Trees},''
  {\em Adv. in Math.}, vol.~139, no.~2, pp.~293 -- 309, 1998.

\bibitem{MalReu}
C.~Malvenuto and C.~Reutenauer, ``{Duality between Quasi-Symmetrical Functions
  and the Solomon Descent Algebra},'' {\em J. of Algebra}, vol.~177, no.~3,
  pp.~967 -- 982, 1995.

\bibitem{Hivert_thesis}
F.~Hivert, {\em {Combinatoire des fonctions quasi-symétiques}}.
\newblock PhD thesis, Informatique fondamentale Université de
  Marne-la-Vallée, 1999.

\bibitem{foissy_2007}
L.~Foissy, ``Bidendriform bialgebras, trees, and free quasi-symmetric
  functions,'' {\em Journal of Pure and Applied Algebra}, vol.~209, no.~2,
  p.~439–459, 2007.

\bibitem{Vargas_thesis}
Y.~Vargas, {\em {Algèbres de Hopf de mots tassés et de fonctions motifs}}.
\newblock PhD thesis, Université du Québec à Montréal, 2019.

\bibitem{berzag}
N.~Bergeron and M.~Zabrocki, ``{The Hopf algebras of symmetric functions and
  quasisymmetric functions in non-commutative variables are free and cofree},''
  2005.

\bibitem{novthi_2006}
J.~C. Novelli and J.~Y. Thibon, ``Polynomial realizations of some
  trialgebras,'' 2006.

\bibitem{ronco_2000}
M.~Ronco, ``Primitive elements in a free dendriform algebra,'' {\em Contemp.
  Math.}, vol.~267, 01 2000.

\bibitem{chapoton_2002}
F.~Chapoton, ``{Un théorème de Cartier–Milnor–Moore–Quillen pour les
  bigèbres dendriformes et les algèbres braces},'' {\em J. of Pure and
  Applied Algebra}, vol.~168, no.~1, pp.~1 -- 18, 2002.

\bibitem{foissy_2011}
L.~Foissy, ``{Primitive elements of the Hopf algebra of free quasi-symmetric
  functions},'' {\em Contemporary Mathematics Combinatorics and Physics},
  p.~79–88, 2011.

\end{thebibliography}
\bibliographystyle{ieeetr}

\end{document}